\documentclass[12pt]{article}
\usepackage{amssymb,xcolor,epic}
\usepackage{amsmath}
\usepackage{graphicx}
\usepackage{enumerate}
\usepackage{mathrsfs}
\usepackage{latexsym}
\usepackage{psfrag}
\usepackage{mathrsfs}
\usepackage{multicol}
\usepackage{algorithm}
\usepackage{algpseudocode}
\usepackage{comment}

\newcommand{\qed}{\hfill $\square$}

\newenvironment{proof}{\noindent{\em Proof}.}{\qed\bigskip}

\newtheorem{theorem}{Theorem}[section]
\newtheorem{lemma}[theorem]{Lemma}

\newtheorem{corollary}[theorem]{Corollary}

\newtheorem{question}{Question}[section]
\newtheorem{conjecture}[question]{Conjecture}

\newcommand{\comments}[1]{}

\title{The antimagic orientation problems for graphs obtained by some graph operations}

\author{
Eranda Dhananjaya
\thanks{Department of Applied Mathematics, National Chung Hsing University, Taichung 40227, Taiwan
{\tt Email:edhananjaya1991@gmail.com} supported by MOST-110-2115-M-005 -005 -MY2.}
\and
Wei-Tian Li
\thanks{Department of Applied Mathematics, National Chung Hsing University, Taichung 40227, Taiwan
{\tt Email:weitianli@nchu.edu.tw} supported by MOST-110-2115-M-005 -005 -MY2.}
}

\date{\today}

\begin{document}
 
\maketitle

\begin{abstract}
A simple graph $G$ is said to admit an antimagic orientation if there exist 
an orientation on the edges of $G$ and a bijection from $E(G)$ to $\{1,2,\ldots,|E(G)|\}$ such that the vertex sums of vertices are pairwise distinct, where the vertex sum of a vertex is defined to be the sum of the labels of the in-edges minus that of the out-edges incident to the vertex. It was conjectured by 
Hefetz, M\"{u}tze, and Schwartz~\cite{HMS10} in 2010
that every connected simple graph admits an antimagic orientation. 
In this paper, we prove that the Mycielski construction and the corona product for graphs with some conditions yield graphs satisfying the above conjecture.
\end{abstract}


\section{Introduction}

In this paper, unless we particularly mention, all graphs are simple. 
Most of the notation and terminology follows~\cite{W01}. 
For a graph $G$, {\em an antimagic labeling} of $G$ is a bijection from the edge set $E(G)$ to the set  $\{1,2,\ldots,|E(G)|\}$ such that the vertex sums are pairwise distinct, where the vertex sum of a vertex $v\in V(G)$ is  defined as the sum of the labels of all edges incident to $v$. It was conjectured by Hartsfield and Ringel~\cite{HR90} in 90's that every connected graph other than $K_2$ has an antimagic labeling. The conjecture has been verified for various graphs, but is still wide open in general. Some remarkable classes of graphs are dense graphs~\cite{AKLRY04}, regular graphs ~\cite{CLPZ16}, and trees with at most one vertex of degree two~\cite{LWZ14}. For more results of the original antimagic problems, we refer a comprehensive survey~\cite{G19} to the readers.

The study of antimagic labeling on directed graphs began much later. 
It was initiated by Hefetz, M\"{u}tze, and Schwartz~\cite{HMS10} in 2010. 
For a directed graph $G$, the vertex sum of a vertex $v$ is defined as 
the sum of the labels of all edges entering $v$ minus the sum of the labels of all edges leaving $v$. 
As before, we say $G$ is antimagic if the vertex sums are pairwise distinct. 
Hefetz et al. pointed out that the directed $C_3$ and $P_3$ are not antimagic. On the other hand, it is not hard to show among all of the directed graphs whose underlying graph is $P_3$ (or $C_3$), there exists a directed graph that is antimagic. So they proposed two questions: 

\begin{question}\label{QUES}
Is every connected directed graph with at least 4 vertices antimagic?
\end{question}

\begin{conjecture}\label{CONJ}
Every connected graph admits an antimagic orientation.
\end{conjecture}
While Question~\ref{QUES} seems to be very difficult, Conjecture~\ref{CONJ} has been studied intensively recently. 
We say an undirected graph $G$ {\em admits an antimagic orientation}, if there exist an orientation $D$ on $E(G)$ and a bijection $\tau$ from $E(G)$ to $\{1,2,\ldots,|E(G)|\}$ such that the vertex sums $s_{(D,\tau)}(v)$, defined earlier in this paragraph, of all vertices are pairwise distinct. 

There are several families of graphs proved to satisfy Conjecture~\ref{CONJ} by different groups of researchers: 
$K_n$ with $n\ge 4$, $S_n$ with $n\ge 3$,  $W_n$ with $n\ge 3$, the $(2d+1)$-regular graphs with $d\ge 0$ by Hefetz et al~\cite{HMS10};
the $2d$-regular graphs with $d\ge 2$ by Li, Song, Wang, Yang, and Zhang~\cite{LSW19}, by Yang~\cite{Y19}, and by Song and Hao~\cite{SH192};  
the biregular graphs by Shan and Yu~\cite{SY17}; 
the Halin graphs by Yu, Chang, and Zhou~\cite{YCZ19}; 
the caterpillars by Lozano~\cite{L18}; 
the lobsters  by Gao and Shan~\cite{GS20};  
and the complete $k$-ary trees by Song and Hao~\cite{SH19}.
In addition to the above special graphs, 
Yang, Carlson, Owens, Perry, Singgih, Song, Zhang, and Zhang~\cite{YCO19}, proved that every connected graph with at least $n\ge 9$ vertices and maximum degree at least $n-5$ admits an antimagic labeling, and 
Song, Yang, and Zhang~\cite{SYZ19} proved that every graph $G$ with independent number at most 4 or least $|V(G)|/2$ admits an antimagic orientation. 
Some of the above results are also true for disconnected graphs.

Let $G$ be a graph with $V(G)=\{v_1,\ldots,v_n\}$. The \emph{Mycielski construction} of $G$, $M(G)$ is a new graph with $V(M(G))=\{v_1,\ldots,v_n\}\cup\{u_1, \ldots,u_n\}\cup\{w\}$ and  $E(M(G))=E(G)\cup\{v_i u_j \mid  \mbox{ if }v_iv_j\in E(G)\}\cup \{wu_i\mid 1\le i\le n\}$. Moreover, we call $\{u_1,\ldots,u_n\}$ the image vertices of $V(G)$. 
An example is $M(K_2)=C_5$. 
This graph operation was introduced by Mycielski~\cite{M55} to construct the graphs having the large chromatic number but the small clique number. 
Given two graphs $G$ and $H$, the {\em corona product} of $G$ and $H$, denoted as $G\odot H$, is the graph obtained by taking one copy of $G$ as the {\em center graph}, $|V(G)|$ copies of $H$ as the {\em outer graphs}, and join the each vertex of $G$ to every vertex of a copy of $H$. This graph product was introduced by Frucht and Harary\cite{FH70} in 1970. 
In general, the product is non-commutative since one can see that $K_1\odot K_2$ is $K_3$ but $K_2\odot K_1$ is $P_4$. 
The research of the antimagic problems for the corona product of graphs can be found in \cite{DIMP15,HASA17,TSE12}.

In this paper, we study the antimagic orientation problem for the Mycielski  construction and also the corona product of graphs. We manage to prove that some graphs obtained by the graph operations mentioned above satisfying Conjecture~\ref{CONJ}.

The organizations of the paper are the following.
In next section, we discuss a technical lemma which is used frequently for constructing the antimagic orientations in recent papers.
The result on the Mycielski constructions is presented in Section 3, while the result on the corona products is presented in Section 4.

\section{The vertex sums induced by the Euler circuits}

A {\em closed walk} in a multiple graph $G$ is an {\em Euler circuit} if it traverses every edge in $G$ exactly once. The following theorem is classical:
\begin{theorem}[Euler, 1736]
A connected multiple graph has an Euler circuit if and only if the degree of every vertex is even.
\end{theorem}
Using the Euler circuit, one can estimate the vertex sum of each vertex induced by the circuit. The lemma below plays an important role in proving our theorems and has appeared in several recent papers (e.g. 
Lemma 2.2 in~\cite{SYZ19} and Lemma 7 in~\cite{S20}). For the sake of completeness, we present the proof in the paper. Let $[a,b]:=\{a,a+1,\ldots,b\}$.

\begin{lemma}\label{EULER}
Let $p\ge 0$ be an integer, and $G$ be a graph with $|E(G)|=m\ge 1$. Either we can find an orientation $D$ of $G$ and a bijection $\tau: E(G)\rightarrow [p+1,,p+m]$ such that for any $v\in V(G)$,
\begin{equation}\label{bd1}
-(p+m)+\lfloor d_G(v)/2\rfloor \le s_{(D,\tau )}(v)\le p+m+ \lfloor d_G(v)/2\rfloor, 
\end{equation}
or an orientation $D'$ and a bijection $\tau'$ such that for any $v\in V(G)$,
\begin{equation}\label{bd2}
-(p+m)- \lfloor d_G(v)/2 \rfloor \le s_{(D',\tau' )}\left(v\right)\le p+m-\lfloor d_G(v)/2\rfloor. 
\end{equation}
\end{lemma}

\begin{proof}
It suffices to prove the lemma for connected graphs.
Let $V(G)=\{v_1,v_2,\ldots,v_n\}$ and $V_1=\{ v\in V(G)\mid d_G(v)\mbox{ is odd} \}$.
By the handshaking theorem, we see that $|V_1|$ is even. If $V_1\neq \varnothing$, let $V_1=\{v_{i_1}, v_{i_2},\ldots, v_{i_{2t}}\}$. 
Define $G^*=G$ when $V_1=\varnothing$, else $G^*$ is obtained by adding $t$ new edges $e^*_1,e^*_2,\ldots,e^*_t$ to $G$, where 
$e^*_j$ is an edge connecting $v_{i_j}$ and $v_{i_{j+t}}$. 
Thus, $d_{G^*}(v_i)$ is even for every vertex $v_i \in G^*$. 
Note that $G^*$ could be a multiple graph.

Pick an Euler circuit $C$ in $G^*$ with the conditions:
If $G$ contains an even degree vertex, we choose the even degree vertex as the initial vertex of $C$. Otherwise, every vertex in $V(G)$ has an odd degree and is incident to exactly one $e^*_i$ in $G^*$. We choose any vertex as the initial vertex of $C$ and the edge $e^*_i$ incident to it as the first edge of $C$. Suppose \[C=u_1,e_1,u_2,e_2,\ldots ,u_{m+t},e_{m+t},u_1,\]
where the $u_i$'s are vertices in $V(G)$ which may have repetitions, and
$e_1, e_2,\ldots, e_{m+t}$ are exactly the edges in $E(G^*)$. 

Let $e_{i_1},e_{i_2},\ldots,e_{i_m}$ in $C$
be the edges in $E(G)$. 
By orienting each edge a direction the same as it is in $C$, we obtain an orientation $D$ of $G$. 
 Observe that $|d^+_D(v_{i})-d^-_D(v_{i})|\le 1$. Define $\tau : E(G)\rightarrow [p+1, p+m]$ with $\tau(e_{i_j})=p+m+1-j$ for all $1\le j\le m$.
 If two edges $e$ and $e'$ incident to $v_i$ appears consecutively in $C$, then $\tau(e)+\tau(e')$ contributes one to the vertex sum of $v_i$. Thereby we calculate the vertex sums as follows.
For any vertex $v_i$ with  $d^+_D(v_{i})=d^-_D(v_{i})=d_G(v_i)/2$, 
\[
s_{(D,\tau)}(v_i)=
\left\{
\begin{array}{ll}
-m+d_G(v_i)/2     &  \mbox{if }v_i\mbox{ is incident to both the edges}\\
&\mbox{labeled with }p+m 
\mbox{ and }p+1;\\
d_G(v_i)/2     & \mbox{otherwise}.
\end{array}
\right.
\]
For any vertex $v_i$ with  $d^+_D(v_{i})-d^-_D(v_{i})=1$, 
let $e_{v_i}$ be the edge leaving $v_i$, whose predecessor in $C$ is some $e^*_j$. We have
\[s_{(D,\tau)}(v_i)=- \tau(e_{v_i})+(d_G(v_i)-1)/2 \]
\comments
{
\[
s_{(D,\tau)}(v_i)=\left\{
\begin{array}{ll}
- \tau(e_i)-m+(d_G(v_i)-1)/2 , &v_i\mbox{ is incident to both the edges}\\
&\mbox{labeled with $p+m$ and $p+1$ };\\
- \tau(e_i)+(d_G(v_i)-1)/2 , &\mbox{otherwise}.
\end{array}
\right.
\]
}
Similarly, for any vertex $v_i$ with  $d^+_D(v_{i})-d^-_D(v_{i})=-1$, we have 
\[s_{(D,\tau)}(v_i)= \tau(e_{v_i})+(d_G(v_i)-1)/2 \]
where $e_{v_i}$ is the edge entering $v_i$, whose successor in $C$ is some $e^*_j$.
Now the lower and upper bounds of Inequality~(\ref{bd1}) are obtained from the vertex sum of a vertex of odd degree.

Let $\tau'=\tau$ and $D'$ be the orientation obtained by reversing the direction of each edge in $D$. By the  similar arguments, we see that $\tau'$ and $D'$ satisfy  Inequality~(\ref{bd2}).
\end{proof}

\section{Mycielski construction of graphs}

Conventionally, the vertex sum of an isolated vertex is defined to be zero. Thus, a graph with more than one isolated vertex does not admit an antimagic orientation. 
Also, an isolated vertex of a graph is still an isolated vertex of its Mycielski construction. So we consider the graphs containing at most one isolated vertex in the following the theorem.

\begin{theorem}\label{MAIN}
Let $G$ be a graph with at most one isolated vertex. Then the Mycielski construction of $G$, $M(G)$
admits an antimagic orientation.  
\end{theorem}

\begin{proof}
First if $G$ contains exactly an isolated vertex, then one can show the independence number of $M(G)$ is $(|V(M(G))|+1)/2$ and one of the maximum independent set is the set consisting of the isolated vertex and all image vertices. So the existence of the antimagic orientation follows the result in~\cite{SYZ19}.

Now suppose that $G$ contains no isolated vertex. 
Let $V(G)=\{v_1,v_2,\ldots,v_n\}$.
Define $V(M(G))=\{v_1,\ldots,v_n\}\cup\{u_1, \ldots,u_n\}\cup\{w\}$ and  $E(M(G))=E(G)\cup\{v_i u_j \mid  \mbox{ if }v_iv_j\in E(G)\}\cup \{wu_i\mid 1\le i\le n\}$. 
Let $|E(G)|=m$. Then $|E(M(G))|=3m+n$. We partition the label set $[1,3m+n]$ into $[1,m]$, $[m+1,3m-n]$, $[3m-n+1,3m]$, and $[3m+1,3m+n]$.

First apply Lemma~\ref{EULER} to obtain an orientation $D_1$ on $E(G)$ and a bijection $\tau$ from $E(G)$ to $[1,m]$ so that $ s_{(D_1,\tau )}\left(v\right)\le m+ (d_G(v_i)-1)/2$ for each $v_i\in V(G)$. 
We extend the orientation $D_1$ to $D$ on $E(M(G))$.   
Since every edge in $E(M(G))\setminus E(G)$ is incident to some image vertex $u_i$, 
we define the direction of it to be entering the $u_i$'s. 
For each vertex $v_i$, we select an edge in $E(M(G))\setminus E(G)$ incident to it as $e_{v_i}$. 
We arbitrarily label the edges of the form $v_iu_j$, except for the $e_{v_i}$'s, with labels in $[m+1,3m-n]$. 
Let $s'_{(D,\tau)}(v_i)$ be the partial (lacking of the label of $e_{v_i}$) vertex sum of $v_i$. 
Without loss of generality, we may assume 
\[s'_{(D,\tau)}(v_1)\le s'_{(D,\tau)}(v_2)\le \cdots\le s'_{(D,\tau)}(v_n). \]
Then we label $e_{v_i}$ with $3m+1-i$ for $1\le i\le n$.
Since $e_{v_{i}}$ is leaving $v_i$, we have $s_{(D,\tau)}(v_i)=s'_{(D,\tau)}(v_i)-3m-1+i$ for $1\le i\le n$,
which implies $s_{(D,\tau)}(v_i)\neq s_{(D,\tau)}(v_j)$ for $i\ne j$.
Finally, for the image vertices $u_1,u_2,\ldots,u_n$, 
suppose we have 
the partial (lacking of the label of $wu_i$) vertex sums satisfying the following:
\[s'_{(D,\tau)}(u_{i_1})\le s'_{(D,\tau)}(u_{i_2})\le \cdots\le s'_{(D,\tau)}(u_{i_n}). \]
Then we label $wu_{i_j}$ with $3m+j$ for $1\le j\le n$.
Since $wu_{i_j}$ is entering $u_{i_j}$, we have $s_{(D,\tau)}(u_{i_j})=s'_{(D,\tau)}(u_{i_j})+3m+j$ for $1\le j\le n$,
which implies $s_{(D,\tau)}(u_{i_j})\neq s_{(D,\tau)}(u_{i_k})$ for $j\ne k$.

Our orientation $D$ gives $s_{(D,\tau)}(u_{i})>0$ for $1\le i\le n$.
Let us see $s_{(D,\tau)}(v_{i})<0$ for $1\le i\le n$. By the fact that every edge of the form $v_iu_j$ is leaving $v_i$ and by 
Lemma~\ref{EULER}, we have 
\begin{eqnarray*}
s_{(D,\tau)}(v_{i}) &\le&  
s_{(D_1,\tau)}(v_i)-d_G(v_i)(m+1)\\
&\le& m+\lfloor d_G(v_i)/2\rfloor-d_G(v_i)(m+1)\\
&<& 0.
\end{eqnarray*}
It remains to show $s_{(D,\tau)}(v_{i})\neq s_{(D,\tau)}(w)$ for $1\le i\le n$.
For each $v_{i}$, there are $d_G(v_i)$ edges leaving $v_i$ and entering the image vertices, and one of them is $e_{v_i}$. The largest possible sum of the labels of these edges is $(\sum_{j=1}^{d_G(v_i)-1}3m-n+1-j)+3m$. 
Thus, 
\begin{eqnarray*}
s_{(D,\tau)}(v_{i}) &\ge&  
s_{(D_1,\tau)}(v_i)-\left(\sum_{j=1}^{d_G(v_i)-1}3m-n+1-j\right)-3m\\
&\ge& -m+ \lfloor d_G(v_i)/2\rfloor-\left(\sum_{j=1}^{n-2}3m-n+1-j\right)-3m\\
&>&-\sum_{j=1}^n 3m+j=s_{(D,\tau)}(w).
\end{eqnarray*}
Therefore the proof is completed. 
\end{proof}

\section{Corona product of graphs}

In this section, we investigate the corona product of graphs. 
Let us denote the center and outer graphs of a corona product by $G$ and $H$. 
The number of vertices and edges of $G$ are denoted as $n$ and $m$, and the numbers of vertices and edges of $H$ are denoted as $k$ and $\ell$. 
Let $V(G)=\{v_1,v_2,\ldots,v_n\}$ and 
$V(H)=\{u_1,u_2,\ldots u_k\}$. 
Denote the $i^{th}$ copy of $H$ as $H^{(i)}$ and the
$i^{th}$ copy of $u_j$ in $H^{(i)}$ as $u_j^{(i)}$. 
We will give the conditions to the outer graphs.

\begin{theorem}\label{corona}
Let $G$ and $H$ be the center and outer graphs of a corona product. Suppose that $H$ contains an even degree vertex. Then 
$G\odot H$ admits an antimagic orientation.
\end{theorem}

\begin{proof}
We construct the orientation of $D$ and the bijection $\tau$ for $G\odot H$ by three steps. 
Since $|E(G\odot H)|=m+n\ell+nk$, we partition the label set $[1, m+n\ell+nk]$ into three sets  $[1,m]$, $[m+1, m+n\ell]$, and $[m+n\ell+1, m+n\ell+nk]$.

First, apply the constructive method in  Lemma~\ref{EULER} to obtain an orientation $D_G$ of $G$ and a bijection $\tau_{G}$ from $E(G)$ to $[1,m]$
such that $s_{(D_G,\tau_G)}(v_i)$ satisfies 
\[
 s_{(D_G,\tau_G)}(v_i)\ge -m+ \lfloor d_G(v_i)/2 \rfloor.
\]
For the outer graph $H$, since $H$ contains an even degree vertex, we also apply Lemma~\ref{EULER} to obtain an orientation $D_H$ of $H$ and a bijection $\tau_H$ from $E(H)$ to $[1,\ell]$ so that 
\[
s_{(D_H,\tau_H)}(u_j)\le  \ell- \lfloor d_H(u_j)/2 \rfloor.
\]

Next, for $1\le i\le n$, orient every $H^{(i)}$ the same as $D_H$ and define $\tau_{H^{(i)}}$ from $E(H^{(i)})$ to $[m+(i-1)\ell+1, m+i\ell]$ by letting
$\tau_{H^{(i)}}(e^{(i)})=\tau_{H}(e)+m+(i-1)\ell$, where $e^{(i)}$ is the copy of $e$ in $H^{(i)}$. 
Note that if $u_j$ is an even degree vertex in $H$, then for $1\le i\le n$, we have 
\[s_{(D_{H^{(i)}}, 
\tau_{H^{(i)}})}(u_j^{(i)})
=
\left\{
\begin{array}{ll}
\ell-d_{H}(u_j)/2,     & \mbox{if }u_j\mbox{ is the initial vertex of the}\\  
&\mbox{Euler circuit traversing }H^*;\\
-d_{H}(u_j)/2,     & \mbox{otherwise},
\end{array}
\right.
\]
which is a constant with respect to $i$.
On the other hand, if $u_j^{(i)}$ is an odd degree vertex in $H^{(i)}$, then   
\[
s_{(D_{H^{(i)}}, \tau_{H^{(i)}})}(u_j^{(i)})
\le 
m+i\ell- \lfloor d_{H}(u_j)/2\rfloor  
\]
for $1\le i\le n$.

Let $M=\{v_iu_j^{(i)}\mid 1\le i\le n,1\le j\le k\}$ be the set of edges connecting $G$ and the $H^{(i)}$'s. We give the direction of every edge in $M$ from $H^{(i)}$ to $G$ and call this orientation $D_M$. 
Then define 
$\tau_M$ from $M$ to $[m+n\ell+1, m+n(k+\ell)]$ as follows.  
First arrange the vertices $u_j^{(i)}$'s 
as $u_{j_1}^{(i_1)},u_{j_2}^{(i_2)},\ldots,u_{j_{nk}}^{(i_{nk})}$, 
satisfying
\[
s_{(D_{H^{(i_1)}}, \tau_{H^{(i_1)}})}(u_{j_1}^{(i_1)})\ge
s_{(D_{H^{(i_2)}}, \tau_{H^{(i_2)}})}(u_{j_2}^{(i_2)})
\ge\cdots
\ge s_{(D_{H^{(i_{nk})}}, \tau_{H^{(i_{nk})}})}(u_{j_{nk}}^{(i_{nk})}). 
\]
Then label the edges $v_{i_1}u_{j_1}^{(i_1)},v_{i_2}u_{j_2}^{(i_2)},\ldots,v_{i_{nk}}u_{j_{nk}}^{(i_{nk})}$ with the labels $m+n\ell+1,m+n\ell+2,\ldots,m+n\ell+nk$ accordingly. 

Now, let $D$ be the orientation of $G\odot H$ such that $D|_{G}=D_G$, $D|_{M}=D_M$, and $D|_{H_i}=D_{H_i}$ for $1\le i\le n$, and 
$\tau$ be the bijection from $E(G\odot H)$ to $[1,m+n\ell+nk]$ such that $\tau|_{G}=\tau_G$, $\tau|_{M}=\tau_M$, and $\tau|_{H_i}=\tau_{H_i}$ for $1\le i\le n$. 
We show that the vertex sums 
$s_{D,\tau}(v)$ for $v\in V(G\odot H)$ are pairwise distinct. 
Since each edge of $M$ is directed from $H_i$ to $G$, we have 
\[s_{(D, \tau)}(u_j^{(i)})=
s_{(D_{H^{(i)}}, 
\tau_{H^{(i)}})}(u_j^{(i)})-\tau_M(v_iu_j^{(j)})
\le m+i\ell-  \lfloor d_{H}(u_j)/2 \rfloor  -(m+n\ell+1)
<0
\]for each $u_j^{(i)}$. Moreover, our labeling method for $v_iu_j^{(i)}$'s gives 
\[
s_{(D, \tau)}(u_{j_1}^{(i_1)})>
s_{(D, \tau)}(u_{j_2}^{(i_2)})
>\cdots> s_{(D, \tau)}(u_{j_{nk}}^{(i_{nk})}). 
\]
On the other hand, for each $v_i$, we have 
\[
s_{(D, \tau)}(v_i)=
s_{(D_{G}, 
\tau_{G})}(v_i)+\sum_{j=1}^k\tau_M(v_iu_j^{(i)})\ge -m+ \lfloor d_{G}(v_i)/2 \rfloor +\sum_{j=1}^k(m+n\ell+j)\\
>0.
\]
However, it is possible that $s_{(D,\tau)}(v_i)=s_{(D,\tau)}(v_j)$ for some $i\neq j$ at this moment. 
Recall that $s_{(D_{H^{(i)}}, \tau_{H^{(i)}})}(u_j^{(i)})$ is a constant for $1\le j\le n$ if $d_H(u_j)$ is even. 
Suppose that $u_1$ is an even degree vertex in $H$. 
Assume that 
\[
s_{(D, \tau)}(v_i)-\tau_M(v_iu_1^{(i)})=
s_{(D_{G}, 
\tau_{G})}(v_i)+\sum_{j=2}^k\tau_M(v_iu_j^{(i)})
\]
is non-decreasing in $i$. Then we permute the labels of $v_1u_1^{(1)},v_2u_1^{(2)},\ldots,v_nu_1^{(n)}$ in an increasing order. Thus, the new  vertex sums of $v_i$'s are all distinct. The vertex sums of $u_1^{(i)}$'s also permuted but are still all distinct. Then the proof is  completed. 
\end{proof}

By the handshaking theorem, we have the following corollary of Theorem~\ref{corona}.

\begin{corollary}\label{hodd}
Let $G$ and $H$ be the center and outer graphs of a corona product. Suppose that $H$ contains an odd number of vertices. Then $G\odot H$ admits an antimagic orientation.
\end{corollary}

\begin{theorem}\label{gleh}
Let $G$ and $H$ be the center and outer graphs of a corona product. If $|V(G)|\le|E(H)|$, then $G\odot H$ admits an antimagic orientation. 
\end{theorem}

\begin{proof}
By Theorem~\ref{corona}, it suffices to prove the statement with the assumption that the degree of each vertex in $H$ is odd.  
We can further assume that the number of vertices of $H$, $k$, is even by Corollary~\ref{hodd}. 
    
For $k=2$, either $H=2K_1$ or $H=K_2$. The former graph contains the even degree vertices, so consider $H=K_2$. Thereby $G=K_1$ since $n\le\ell= 1$. We have $G\odot H=K_3$ which admits an antimagic orientation. 

Suppose $k\ge 4$. We partition the set 
$[1,m+n\ell+nk]$ into four sets 
$[1,n]$, $[n+1,n+m]$, 
$[n+m+1,m+n(\ell+1)]$, and 
$[m+n(\ell+1)+1,m+n\ell+nk]$.
First apply Lemma~\ref{EULER} to construct the orientation $D_G$ and the bijection $\tau_G$ from $E(G)$ to $[n+1,n+m]$ satisfying   
\[
s_{(D_G,\tau_G)}(v_i)\ge -(m+n)- \lfloor d_G(v_i)/2\rfloor,
\]
and the orientation $D_H$ and the bijection $\tau_H$ from $E(H)$ to $[1,\ell]$ satisfying   
\[
s_{(D_{H}, \tau_{H})}(u_j)
\le 
\ell- \lfloor d_{H}(u_j)/2\rfloor.  
\]
Then imitate the proof of Theorem~\ref{corona} to obtain $D_{H^{(i)}}$ and $\tau_{H^{(i)}}$ with 
\[
s_{(D_{H^{(i)}}, \tau_{H^{(i)}})}(u_j^{(i)})
\le 
n+m+i\ell- \lfloor d_{H}(u_j)/2\rfloor.  
\]
Recall that for each $u_j\in V(H)$, its vertex sum $s_{(D_H,\tau_H)}(u_j)$ is equal to 
\[
\tau_H(e_{u_j})-\lfloor d_{H}(u_j)/2\rfloor \mbox{ or }  -\tau_H(e_{u_j})-\lfloor d_{H}(u_j)/2\rfloor,       
\] where $e_{u_j}$ is an edge incident to $u_j$ whose definition is given in Lemma~\ref{EULER}. 
We claim that there exists some $u_j$ with  
\[
s_{(D_H,\tau_H)}(u_j)=\tau_H(e_{u_j})-\lfloor d_{H}(u_j)/2\rfloor>0.
\]
Since the total vertex sum $\sum_{j=1}^k s_{(D_H,\tau_H)}(u_j)=0$, we must have positive $s_{(D_H,\tau_H)}(u_j)$ for some $u_j$ unless $s_{(D_H,\tau_H)}(u_j)=0$ for every $u_j$. 
Therefore, if the claim is not true, then 
$s_{(D_H,\tau_H)}(u_j)=0$ for every $u_j$. 
By the pigeonhole principle, there exist two vertices $u_j$ and $u_{j'}$ that have the same degree with 
\[
\tau_H(e_{u_j})-\lfloor d_{H}(u_j)/2\rfloor
=
\tau_H(e_{u_{j'}})-\lfloor d_{H}(u_{j'})/2\rfloor
=0. 
\]
Thus, $\tau_H(e_{u_j})= \tau_H(e_{u_j'})$ and hence 
$e_{u_j}= e_{u_j'}$. However, if $e_{u_j}= e_{u_j'}$, then the label of $e_{u_j}$ should contribute $\tau_H(e_{u_j})$ and $-\tau_H(e_{u_j})$ to the two endpoints of $e_{u_j}$, respectively. This contradicts our assumption. As a consequence, we must have some $u_j$ with positive $s_{(D_H,\tau_H)}(u_j)$ and the claim is proved. 

Now let $u_1$ be a vertex such that 
$s_{(D_H,\tau_H)}(u_1)>0$ and 
$U=\{u_1^{(i)}\mid 1\le i\le n\}$. 
We define the orientation $D_M$ by directing all the edges in $M$ from $H^{(i)}$ to $G$. To define $\tau_M$, first arrange the vertices $u_j^{(i)}$'s for $1\le i\le n$ and $2\le j\le k$ 
as $u_{j_1}^{(i_1)},u_{j_2}^{(i_2)},\ldots,u_{j_{n(k-1)}}^{(i_{n(k-1)})}$
such that 
\[
s_{(D_{H^{(i_1)}}, \tau_{H^{(i_1)}})}(u_{j_1}^{(i_1)})\ge
s_{(D_{H^{(i_2)}}, \tau_{H^{(i_2)}})}(u_{j_2}^{(i_2)})
\ge\cdots
\ge s_{(D_{H^{(i_{n(k-1)})}}, \tau_{H^{(i_{n(k-1)})}})}(u_{j_{n(k-1)}}^{(i_{n(k-1)})}). 
\]
Then label the edges $v_{i_1}u_{j_1}^{(i_1)},v_{i_2}u_{j_2}^{(i_2)},\ldots,v_{i_{n(k-1)}}u_{j_{n(k-1)}}^{(i_{n(k-1)})}$ with the labels $m+n(\ell+1)+1,m+n(\ell+1)+2,\ldots,m+n(k+\ell)$ accordingly. 
Now for each vertex $v_i$, we consider the partial vertex sum $s'_{(D,\tau)}(v_i)=s_{(D_G,\tau_G)}(v_i)+\sum_{j=2}^k\tau_M(v_iu_j^{(i)})$. Arrange the $v_i$'s so that
\[
s'_{(D,\tau)}(v_{i_1})\le s'_{(D,\tau)}(v_{i_2})\le\cdots\le s'_{(D,\tau)}(v_{i_n}).
\]
Then label the edges $v_{i_1}u_1^{({i_1})},v_{i_2}u_1^{({i_2})},\ldots, v_{i_n}u_1^{({i_n})}$
with the labels $1,2,\ldots,n$ accordingly to complete defining $\tau_M$. 

The vertex sums $s_{(D,\tau)}(u_{j}^{(i)})$'s for $1\le i\le n$ and $2\le j\le k$ are all distinct since 
\[
s_{(D_{H^{(i_1)}}, \tau_{H^{(i_1)}})}(u_{j_1}^{(i_1)})>
s_{(D_{H^{(i_2)}}, \tau_{H^{(i_2)}})}(u_{j_2}^{(i_2)})
>\cdots
> s_{(D_{H^{(i_{n(k-1)})}}, \tau_{H^{(i_{n(k-1)})}})}(u_{j_{n(k-1)}}^{(i_{n(k-1)})})
\]
by our labeling method. Moreover, we have 
\begin{eqnarray*}
s_{(D, \tau)}(u_j^{(i)})&=&
s_{(D_{H^{(i)}}, \tau_{H^{(i)}})}(u_j^{(i)})
-\tau(v_iu_j^{(i)})\\
&\le& n+m+i\ell- \lfloor d_{H}(u_j)/2\rfloor-[  m+n(\ell+1)+1]<0. 
\end{eqnarray*}
Analogously, for the vertices in $V(G)$, we have $s_{(D,\tau)}(v_{i_1})< s_{(D,\tau)}(v_{i_2})<\cdots< s_{(D,\tau)}(v_{i_n})$,  and when $k\ge 4$, 
\begin{eqnarray*}
s_{(D,\tau)}(v_{i_1})&=&s_{(D_G,\tau_G)}(v_{i_1})+\sum_{j=2}^k\tau_M(v_{i_1}u_j^{({i_1})})+1\\
&\ge& -(n+m)-\lfloor d_G(v_{i_1})/2\rfloor +\sum_{j=2}^k\tau_M(v_{i_1}u_j^{({i_1})})+1\\
&>& -(n+m)-m/2+\sum_{j=2}^4[m+n(\ell+1)+j-1]\\
&>&m+n(\ell+1)+1.
\end{eqnarray*}
For each vertex in $U$, we have 
\begin{eqnarray*}
s_{(D,\tau)}(u_1^{(i)})&=&s_{(D_{H^{(i)}},\tau_{H^{(i)}})}(u_1^{(i)})-\tau_M(v_iu_1^{(i)})\\
&= &\tau_{H^{(i)}}(e_{u_1}^{(i)})-\lfloor d_H(u_1)/2\rfloor-\tau_M(v_iu_1^{(i)}) \\
&\ge &\tau_H(e_{u_1})+n+m+(i-1)\ell-\lfloor d_H(u_1)/2\rfloor-n\\
&>&0.
\end{eqnarray*}
On the other hand, for $1\le i\le n$
\[
s_{(D,\tau)}(u_1^{(i)})=\tau_{H^{(i)}}(e_{u_1}^{(i)})-\lfloor d_H(u_1)/2\rfloor-\tau_M(v_iu_1^{(i)}) <\tau_{H^{(i)}}(e_{u_1}^{(i)})\le m+n+n\ell.
\]
Hence the vertex sum of each vertex in $U$ is different to that of any vertex not in $U$. 
To finish the proof, we verify that all vertices in $U$  have distinct vertex sums. For $i\neq i'$, we have 
\begin{eqnarray*}
&&|s_{(D,\tau)}(u_1^{(i)})-s_{(D,\tau)}(u_1^{(i')})|\\
&=&\left|\left[s_{(D_{H^{(i)}},\tau_{H^{(i)}})}(u_1^{(i)})-\tau_M(v_iu_1^{(i)})\right]-\left[s_{(D_{H^{(i')}},\tau_{H^{(i')}})}(u_1^{(i')})-\tau_M(v_iu_1^{(i')})\right]\right|\\
&\ge & |s_{(D_{H^{(i)}},\tau_{H^{(i)}})}(u_1^{(i)})-s_{(D_{H^{(i')}},\tau_{H^{(i')}})}(u_1^{(i')}) |- |\tau_M(v_{i}u_1^{(i)})-\tau_M(v_{i'}u_1^{(i')}) |\\
&\ge &|(i-i')\ell|-(n-1)\\
&>&0
\end{eqnarray*}
by the assumption $\ell\ge n$. So the proof is completed. 
\end{proof}


\begin{thebibliography}{99}


\bibitem{AKLRY04}N. Alon, G. Kaplan, A. Lev, Y. Roditty, and R. Yuster, Dense graphs are antimagic, Journal of Graph Theory, 47 (2004), 297-309.



\bibitem{CLPZ16} F.-H. Chang, Y.-C. Liang, Z. Pan, and X. Zhu, Antimagic labeling of regular graphs, Journal of Graph Theory, 82 (2016), 339-349.



 \bibitem{DIMP15} J. W. Daykin, C. S.  Iliopoulos, M. Miller, O. Phanalasy, Antimagicness of Generalized Corona and Snowflake Graphs. Math. Comput. Sci. 9, 105–111 (2015). 


\bibitem{FH70}
R. Frucht and F. Harary, On the corona of two graphs, Aequationes Math. 4
(1970), 322–325.

\bibitem{G19} J. A. Gallian, 
A Dynamic Survey of Graph Labeling, 
Electronic Journal of Combinatorics (2019), \#DS6 

\bibitem{GS20}
Y. Gao and  S. Shan, 
Antimagic orientation of lobsters,
Discrete Applied Mathematics, 287 (2020), 21-26.


\bibitem{HASA17}
A. K. Handa, A. Godinho, T. Singh, and S. Arumugam, Distance antimagic labeling of join and corona of two graphs, AKCE International Journal of Graphs and Combinatorics (2017), 14(2), 172-177.


\bibitem{HMS10}
D. Hefetz, T. M\"{u}tze, and J. Schwartz,  
On antimagic directed graphs,  
Journal of Graph Theory, 64 (2010), 219–232.


\bibitem{HR90} N. Hartsfield and G. Ringel, Pearls in Graph Theory, Academic Press, INC., Boston 1990.



\bibitem{LSW19} 
T. Li, Z.-X. Song, G. Wang , D. Yang, and C.-Q. Zhang, Antimagic orientations of even regular graphs,  
Journal of Graph Theory, 90 (2019), 46–53.



\bibitem{LWZ14} Y.-C. Liang, T. Wong, and X. Zhu, { Anti-magic labeling of trees}, Discrete Mathematics, 331 (2014), 9-14.


\bibitem{L18}A. Lozano, 
Caterpillars have antimagic orientations,
Analele Universitatii "Ovidius" Constanta - Seria Matematica, 26(3) (2018), 171–180.


\bibitem{M55}
J. Mycielski, 
Sur le coloriage des graphes, 
Colloquium Mathematicum 3 (1955), 161-162

\bibitem{S20}
S. Shan, 
Antimagic orientation of graphs with minimum degree at least 33, 
Journal of Graph theory, 98(4), (2021), 676--690.


\bibitem{SY17} 
S. Shan and X. Yu, 
Antimagic orientation of biregular bipartite graphs, 
Electronic Journal of Combinatorics, 24(4) (2017), paper P4.31.


\bibitem{SH19}
C. Song and R. X. Hao, 
Antimagic orientations for the complete $k$-ary trees,
Journal Combinatorial Optimization, 38 (2019), 1077–1085.

\bibitem{SH192}
C. Song and R. X. Hao, 
Antimagic orientations of disconnected even regular graphs,
Discrete Mathematics, 342(8) (2019), 2350-2355.

\bibitem{SYZ19}
Z.-X. Song, D. Yang, and F. Zhang, 
Antimagic orientations of graphs with given independence number, 
Discrete Applied Mathematics, 291, (2021) 163--160.


\bibitem{TSE12}M. I. Tilukay, A. N. M. Salman, and M. Elviyenti, On super $d$-face antimagic total labelings of the corona product of a tree with $r$ copies of a path, AIP Conference Proceedings 1450, 218 (2012).


\bibitem{W01} D. B. West, 
Introduction to Graph Theory, 2nd edition. Pearson, 2001.


\bibitem{Y19}
D. Yang, 
A note on anti-magic orientation of even regular graphs, 
Discrete Applied Mathematics, 267 (2019), 224-228.



\bibitem{YCO19} 
D. Yang, J. Carlson, A. Owens, K. E. Perry, I. Singgih, Z.-X. Song, F. Zhang, and X. Zhang, 
Antimagic orientations of graph with large maximum degree, 
Discrete Mathematics, 343 (2020), 112-123. 


\bibitem{YCZ19}X. Yu, Y. Chang, and S. Zhou, 
Antimagic orientation of Halin graphs, 
Discrete Mathematics, 342 (2019), 3160-3165.





\end{thebibliography}
\end{document}